\documentclass[12pt]{article}%
\usepackage{graphicx}
\usepackage{amsmath}
\usepackage{amsfonts}
\usepackage{amssymb}%
\providecommand{\U}[1]{\protect\rule{.1in}{.1in}}
\newtheorem{theorem}{Theorem}

\newtheorem{definition}[theorem]{Definition}

\newtheorem{lemma}[theorem]{Lemma}

\newtheorem{proposition}[theorem]{Proposition}
\newtheorem{remark}[theorem]{Remark}

\newenvironment{proof}[1][Proof]{\noindent\textbf{#1.} }{\ \rule{0.5em}{0.5em}}
\begin{document}

\title{\textbf{ The major index }$\left(  \mathsf{maj}\right)  $\textbf{ and its
Sch\"{u}tzenberger dual}}
\author{Oleg Ogievetsky$^{\sharp,\flat,\dag}$ \& Senya Shlosman$^{\natural
,\sharp,\flat,\ddag}$\\$^{\natural}$Krichever Center for Advance Studies, Moscow, Russia;\\$^{\sharp}$Aix Marseille Univ, Universite de Toulon, \\CNRS, CPT, Marseille, France;\\$^{\flat}$Inst. of the Information Transmission Problems, \\RAS, Moscow, Russia;\\$^{\dag}$Lebedev Physical Institute, Moscow, Russia,\\$^{\ddag}$BIMSA, Beijing, China\\Oleg.Ogievetsky@gmail.com, shlosman@gmail.com}
\maketitle

\begin{abstract}
We construct the independent particle representation for the Semistandard
Young Tableaux (SsYT) of skew shape $\lambda/\mu.$ The partition function of
this particle system gives the generating function of the SsYT of skew shape
$\lambda/\mu.$ Thus we obtain a bijective proof of the Stanley formula for the
SsYT generating function.

To do this we define for every SsYT $T$ its plinth, $\mathsf{p}\left(
T\right)  ,$ which is a SsYT of the same shape $\lambda/\mu.$ The set of
plinths is finite. Our bijection associates to every SsYT $T$ a pair $\left(
\mathsf{p}\left(  T\right)  ,Y\left(  T-\mathsf{p}\left(  T\right)  \right)
\right)  ,$ where $Y\left(  T-\mathsf{p}\left(  T\right)  \right)  $ is the
reading Young diagram of the SsYT $\left(  T-\mathsf{p}\left(  T\right)
\right)  $. \newline In particular, every Standard Young Tableau (SYT) $P$ has
its plinth, $\mathsf{p}\left(  P\right)  $. The two statistics of SYT-s -- the
volume $\left\vert \mathsf{p}\left(  P\right)  \right\vert $ and
$\mathsf{maj}\left(  P\right)  $ -- are related via the Sch\"{u}tzenberger
involution $Sch:$%
\[
\left\vert \mathsf{p}\left(  P\right)  \right\vert =\mathsf{maj}\left(
Sch\left(  P\right)  \right)  .
\]

\end{abstract}

\section{Introduction}

We first remind the reader about the description of the integer partitions as
a particle system. The partition function of this particle system gives the
generating function of the integer partitions. We then describe our main
result, which provides analogous description for the Semistandard Young Tableaux.

\subsection{Independent variables representation}

Consider the set of partitions $\mathcal{Y}_{n},$ which consists of all Young
diagrams $Y$ with $n$ columns:%
\[
Y=\left\{  \left(  y_{1},...,y_{n}\right)  :y_{i}\in\mathbb{Z}^{1},0\leq
y_{1}\leq...\leq y_{n}\right\}  .
\]
We put $\left\vert Y\right\vert =y_{1}+...+y_{n},$ and if $\left\vert
Y\right\vert =N$ then we say that $Y$ is a partition of the integer $N$ into
$n$ parts. We want to get an expression for the generating function
\[
G_{n}\left(  q\right)  =\sum_{Y\in\mathcal{Y}_{n}}q^{\left\vert Y\right\vert }%
\]
(which is well known, of course). The easiest way to do it is via the
reasization that $G_{n}\left(  q\right)  $ is the partition function of the
following simple particle system:

Consider the set of $n$ boxes $\mathsf{b}_{1},...,\mathsf{b}_{n},$ each
$\mathsf{b}_{k}$ containing some number $\xi_{k}$ of particles, $\xi_{k}%
\geq0,$ $k=1,...,n.$ The particles are not interacting; the only non-trivial
feature of the system is that the activities of particles depend on the box;
in the box $\mathsf{b}_{k}$ the corresponding activity is $z_{k}.$ Thus the
grand canonical partition function of this ensemble is%
\[
Z\left(  z_{1},...,z_{n}\right)  =\prod_{k=1}^{n}\left(  1+z_{k}+z_{k}%
^{2}+...\right)  =\prod_{k=1}^{n}\left(  \frac{1}{1-z_{k}}\right)  .
\]
The important observation is that if we take $z_{k}=q^{k},$ then
\begin{equation}
Z\left(  q,q^{2},...,q^{n}\right)  =G_{n}\left(  q\right)  . \label{29}%
\end{equation}
To see this consider the following bijection $B_{0}$ between the set $\Xi
_{n}=\left\{  \xi_{1},...,\xi_{n}\right\}  $ of particle configurations and
the set of partitions $\mathcal{Y}_{n}:$ the image $Y=\left\{  y_{1}%
,...,y_{n}\right\}  =B_{0}\left\{  \xi_{1},...,\xi_{n}\right\}  $ of the
configuration $\left\{  \xi_{1},...,\xi_{n}\right\}  $ is defined by:%
\begin{align*}
y_{1}  &  =\xi_{n},\\
y_{2}-y_{1}  &  =\xi_{n-1},\\
&  ...\\
y_{n}-y_{n-1}  &  =\xi_{1.}%
\end{align*}
Since $y_{1}+y_{2}+...+y_{n}=\xi_{1}+2\xi_{2}+...+n\xi_{n},$ we have
\[
q^{y_{1}+y_{2}+...+y_{n}}=q^{\xi_{1}}\left(  q^{2}\right)  ^{\xi_{2}%
}...\left(  q^{n}\right)  ^{\xi_{n}},
\]
and our claim $\left(  \ref{29}\right)  $ follows. Thus we recover the
well-known answer:
\[
G_{n}\left(  q\right)  =\sum_{Y\in\mathcal{Y}_{n}}q^{\left\vert Y\right\vert
}=\prod_{k=1}^{n}\left(  \frac{1}{1-q^{k}}\right)  .
\]

As a result, the study of the statistics of the partitions with $n$ parts is
possible via the study of the independent (though differently distributed)
random variables $\xi_{1},...,\xi_{n}.$ This representation of the partitions
via independent random particles is very helpful, as is demonstrated for
example in \cite{DVZ} (from where we have borrowed the title of this subsection).

\subsection{Independent variables representation, 2D case}

The above partitions $Y$ are 1D integer arrays. One of the main goal of this
work is to get a similar independent variables representation for 2D arrays of
integers, such as plane partitions or semistandard Young tableaux.

For example, let $T=\left\{  t_{i,j}:1\leq i\leq n,1\leq j\leq m\right\}  $ be
a rectangular array of non-negative integers. It is called a reverse plane
partition if $t_{i,j}\leq t_{i+1,j}$ for all $1\leq i\leq n-1,1\leq j\leq m$
and $t_{i,j}\leq t_{i,j+1}$ for all $1\leq i\leq n,1\leq j\leq m-1.$ Let us
denote by $\mathcal{P}_{n,m}^{RPP}$ the set of all reverse plane partitions
$T$ sitting on the rectangle $n\times m.$ As before, we are interested in the
generating function
\begin{equation}
G_{n,m}^{RPP}\left(  q\right)  =\sum_{T\in\mathcal{P}_{n,m}^{RPP}%
}q^{\left\vert T\right\vert }, \label{79}%
\end{equation}
where $\left\vert T\right\vert =\sum t_{i,j},$ and in the representation of
this ensemble as a particle system the generating function $\left(
\ref{79}\right)  $ will be the partition function of this particle system.

Clearly, there exists the \textit{reading map} $r:\mathcal{P}_{n,m}%
^{RPP}\rightarrow\mathcal{Y}_{nm},$ which to every 2D partition $T=\left\{
t_{i,j}\right\}  $ corresponds the (1D) partition $Y=r\left(  T\right)  ,$
composed by the $nm$ integers $t_{i,j}$ listed in non-decreasing order.
However, the corresponding collection $\xi_{1},...,\xi_{nm}$ of particles is
not sufficient for describing the ensemble $\mathcal{P}_{n,m}^{RPP},$ since
the map $r$ is not a bijection, of course, but is a several-to-one map. It
turns out that one needs to add one extra variable, $\pi_{n,m},$ independent
of $\xi_{1},...,\xi_{nm},$ to get the desired representation. This varialbe
$\pi_{n,m}$ is called a \textit{pedestal}. Pedestals were introduced in
\cite{S}. The pedestal $\pi_{n,m}\left(  T\right)  $ is again a reversed plane
partition, $\pi_{n,m}\left(  T\right)  \in\mathcal{P}_{n,m}^{RPP},$ but the
range of the map $\pi_{n,m}:\mathcal{P}_{n,m}^{RPP}\rightarrow\mathcal{P}%
_{n,m}^{RPP}$ is finite. This finite set of all pedestals is denoted by
$\Pi_{n,m}\subset\mathcal{P}_{n,m}^{RPP}.$ The extra pedestal variable is
enough for our purposes: there exists a bijection $b:\mathcal{P}_{n,m}%
^{RPP}\rightarrow\Pi_{n,m}\times\mathcal{Y}_{nm}$ which allows one to write
down the generating function $\left(  \ref{79}\right)  $ as a partition
function of the extended particle system $\left\{  \pi_{n,m},\xi_{1}%
,...,\xi_{nm}\right\}  ,$ so in particular we have%
\[
G_{n,m}^{RPP}\left(  q\right)  =\sum_{T\in\mathcal{P}_{n,m}^{RPP}%
}q^{\left\vert T\right\vert }=\left(  \sum_{\pi\in\Pi_{n,m}}q^{\left\vert
\pi\right\vert }\right)  \prod_{k=1}^{nm}\left(  \frac{1}{1-q^{k}}\right)  .
\]
We are omitting here some important details, which can be found in \cite{S}.
The polinomial $\sum_{\pi\in\Pi_{n,m}}q^{\left\vert \pi\right\vert }$ is what
is called pedestal polinomial in \cite{OS}. For more information on pedestals
and pedestal matrices see \cite{KKOPSS}.

In what follows, we will develop a similar program for the case of
semistandart Young tableaux. We will find the particle system representation
for them, thus getting the generating function of the semistandart Young
tableaux via computing the partition function of this extended particle
system. An extra particle which appear there is called \textit{plinth}. We
then study the relation between the plinths and the known index $\mathsf{maj}%
\left(  \ast\right)  $ of the standard Young tableaux.

\subsection{Particle system for the semistandard Young tableaux}

We will consider a more general situation than above.

Let $\lambda$ be a Young diagram, having $\ell\left(  \lambda\right)  $ rows,
of lengths $\lambda_{1}\geq\lambda_{2}\geq...\lambda_{\ell\left(
\lambda\right)  }$. Let $\mu\subset\lambda$ be a smaller Young diagram. The
skew diagram, $\lambda/\mu,$ is the complement, $\lambda\setminus\mu$. The
cells $\left(  ij\right)  $ of $\lambda/\mu$ are pairs of integers,
satisfying:  $1\leq i\leq\ell\left(  \lambda\right)  ,\ \mu_{i}<j\leq
\lambda_{i}.$ A Semistandard Young tableau (SsYT) $T$ of shape $\lambda/\mu,$
is an array $T=\left(  T_{ij}\right)  $ of \textbf{non-negative} integers of
shape $\lambda/\mu$, which are weakly increasing in every row (i.e. in $j$)
and strictly increasing in every column of $\lambda/\mu$ (i.e. down). Denote
the (infinite) set of all SsYT of shape $\lambda/\mu$ by $\mathcal{P}%
_{\lambda/\mu}^{Ss}.$ A Standard Young tableau (SYT) $P$ of shape $\lambda
/\mu$ is an array $P=\left(  P_{ij}\right)  $ of \textbf{positive} integers
$\left\{  1,2,...,n=\left\vert \lambda/\mu\right\vert \right\}  $ of shape
$\lambda/\mu$ , each integer taken once, which is strictly increasing in every
row and in every column. Denote the (finite) set of all SYT of shape
$\lambda/\mu$ by $SYT\left(  \lambda/\mu\right)  .$

The generating function $G_{\lambda/\mu}^{Ss}\left(  q\right)  $ of the
Semistandard Young Tableaux (SsYT) of shape $\lambda/\mu\ $is a specialization
of the well-known \textit{Schur function} $s_{\lambda/\mu}$:
\[
G_{\lambda/\mu}^{Ss}\left(  q\right)  =\sum_{T\in\mathcal{P}_{\lambda/\mu
}^{Ss}}q^{vol\left(  T\right)  }=s_{\lambda/\mu}\left(  1,q,q^{2},...\right)
.
\]

We will construct a  particle system $\left\{  \mathsf{p},\xi_{1},...,\xi
_{n}\right\}  $ which has $G_{\lambda/\mu}^{Ss}\left(  q\right)  $ as its
partition function. The extra variable $\mathsf{p}$ is called a
\textit{plinth, }it takes values in some finite subset $\mathcal{N}%
_{\lambda/\mu}\subset\mathcal{P}_{\lambda/\mu}^{Ss}.$ We will also exhibit a
bijection, $B^{Ss}:\mathcal{P}_{\lambda/\mu}^{Ss}\rightarrow\mathcal{N}%
_{\lambda/\mu}\times\mathcal{Y}_{n},$ respecting the volumes, so that we have
\begin{equation}
G_{\lambda/\mu}^{Ss}\left(  q\right)  =\left(  \sum_{\mathsf{p}\in
\mathcal{N}_{\lambda/\mu}}q^{vol\left(  \mathsf{p}\right)  }\right)
\prod_{k=1}^{n}\left(  \frac{1}{1-q^{k}}\right)  .\label{26}%
\end{equation}
The construction of the bijection $B^{Ss}$ is one of our main results. 

In fact, a (different) formula for the generating function $G_{\lambda/\mu
}^{Ss}\left(  q\right)  $ is known. It expresses the generating function of
the SsYT of shape $\lambda/\mu$ via the major index $\mathsf{maj}(\ast)$
(defined in $\left(  \ref{44}\right)  $ below) of Standard Young Tableaux
(SYT) of the same shape $\lambda/\mu$:

\begin{proposition}
(\cite{St}, Proposition 7.19.11) Let $\left\vert \lambda/\mu\right\vert =n.$
Then%
\begin{equation}
G_{\lambda/\mu}^{Ss}\left(  q\right)  =\left(  \sum_{P\in SYT\left(
\lambda/\mu\right)  }q^{\mathsf{maj}\left(  P\right)  }\right)  \prod
_{k=1}^{n}\left(  \frac{1}{1-q^{k}}\right)  \label{21}%
\end{equation}
where $P$ ranges over all SYT of shape $\lambda/\mu.$
\end{proposition}

Moreover, as we explain below, there is a natural one-to-one correspondence
between the plinths $\mathsf{p}\in\mathcal{N}_{\lambda/\mu}$ and the SYT $P\in
SYT\left(  \lambda/\mu\right)  ,$ so every plinth $\mathsf{p}$ can be uniquely
written as $\mathsf{p}\left(  P\right)  $ for an appropriate SYT  $P.$ Looking
again on the formulas above, one sees that
\begin{equation}
\sum_{\mathsf{p}\in\mathcal{N}_{\lambda/\mu}}q^{vol\left(  \mathsf{p}\right)
}=\sum_{P\in\mathsf{S}YT\left(  \lambda/\mu\right)  }q^{vol\left(
\mathsf{p}\left(  P\right)  \right)  }=\sum_{P\in SYT\left(  \lambda
/\mu\right)  }q^{\mathsf{maj}\left(  P\right)  }.\label{23}%
\end{equation}
Yet, the individual $P$-terms in $\left(  \ref{23}\right)  $ are not
coinciding, i.e. the two statistics on SYT-s -- $vol\left(  \mathsf{p}\left(
P\right)  \right)  $ and $\mathsf{maj}\left(  P\right)  $ -- are different! 

This puzzle is resolved by help of the technique developed by
Sch\"{u}tzenberger: it turns out that the Sch\"{u}tzenberger involution $Sch$
of a SYT $P$ has the plinth with the volume equal to $\mathsf{maj}\left(
P\right)  :$%
\begin{equation}
vol\left(  \mathsf{p}\left(  Sch\left(  P\right)  \right)  \right)
=\mathsf{maj}\left(  P\right)  .\label{71}%
\end{equation}
Of course, the Sch\"{u}tzenberger involution, see \cite{Sch}, is defined
initially only for SYT $P$ of the straight shape $\lambda,$ and not for the
skew shapes $\lambda/\mu.$ But its extension to skew shapes is possible, as we
will explain below, due to the beautiful results of Haiman, \cite{H}.

Note that our derivation of the formula $\left(  \ref{26}\right)  $ gives a
bijective proof of the Stanley formula $\left(  \ref{21}\right)  .$

We finish the introduction by an example for $\left(  \ref{71}\right)  $: we
take $\lambda=%
\begin{tabular}
[c]{|l|l|l|}\hline
$\ast$ & $\ast$ & $\ast$\\\hline
$\ast$ & $\ast$ & \\\hline
\end{tabular}
\ \ \ \ ,\ \mu=\varnothing,$ and $P=%
\begin{tabular}
[c]{|l|l|l|}\hline
$1$ & $2$ & $3$\\\hline
$4$ & $5$ & \\\hline
\end{tabular}
\ .$ It turns out that $\mathsf{p}\left(  P\right)  =%
\begin{tabular}
[c]{|l|l|l|}\hline
$0$ & $0$ & $0$\\\hline
$1$ & $1$ & \\\hline
\end{tabular}
,$ while $Sch\left(  P\right)  =%
\begin{tabular}
[c]{|l|l|l|}\hline
$1$ & $2$ & $5$\\\hline
$3$ & $4$ & \\\hline
\end{tabular}
.$ We have%
\[
\left\vert \mathsf{p}\left(  P\right)  \right\vert =2=\mathsf{maj}\left(
\begin{tabular}
[c]{|l|l|l|}\hline
$1$ & $2$ & $5$\\\hline
$3$ & $4$ & \\\hline
\end{tabular}
\ \ \ \ \right)  ,
\]
as claimed (since the cell $2$ is the only \textit{descent} of the SYT $%
\begin{tabular}
[c]{|l|l|l|}\hline
$1$ & $2$ & $5$\\\hline
$3$ & $4$ & \\\hline
\end{tabular}
\ \ \ .$)

The rest of the paper is organized as follows. We define plinths in the next
Section. In the Sections 3 and 4 we remind the reader about the jeu de taquin
and the Sch\"{u}tzenberger involution. Since these two concepts are well
known, we restrict ourselves to just the examples of both, referring the
reader to the Appendix 1 of the Stanley book \cite{St} (written by S. Fomin),
for more information. In the Section 5 we prove $\left(  \ref{71}\right)  $
for Standard Young Tableaux of straight shape. Section 6 contains our
extension of the Sch\"{u}tzenberger involution to the Standard Young Tableaux
of skew shapes. We prove the relation $\left(  \ref{71}\right)  $ for Standard
Young Tableaux of skew shape in the last Section 7.

\section{Plinths}

Let $\lambda/\mu$ be a skew shape with $n$ cells. We denote by $ij$ or
$\left(  ij\right)  $ its cells. A Semistandard Young Tableau (SsYT) of shape
$\lambda/\mu$ is an array $T=\left(  T_{ij}\right)  $ of non-negative integers
of shape $\lambda/\mu$ (i.e., $1\leq i\leq\ell\left(  \lambda\right)  ,$
$\mu_{i}<j\leq\lambda_{i}$), which are weakly increasing in every row and
strictly increasing in every column. We think that $i$ increases to the right,
while $j$ increases down. We denote by $\left\vert T\right\vert $ the sum of
all $T_{ij}.$

Writing down the integers $T_{ij}$ in non-decreasing order, $\left(
y_{1},...,y_{n}\right)  $ provides a partition of the number $\left\vert
T\right\vert $ (with possibly zero parts). We denote the corresponding Young
diagram -- the reading of $T$ -- by $Y\left(  T\right)  .$

Let $Q$ be a Standard Young Tableau (SYT) of shape $\lambda/\mu.$ For every
$k=1,...,n$ there is a unique cell $\left(  ij\right)  $ for which $Q_{ij}=k.$
In that case we will say that $\left(  ij\right)  =Q\left(  k\right)  .$ The
value $k,$ $k=1,...,n-1$ (and the cell $Q\left(  k\right)  $) is called a
descent of $Q,$ if
\begin{equation}
Q\left(  k\right)  =\left(  ij\right)  ,\ Q\left(  k+1\right)  =\left(
i^{\prime}j^{\prime}\right)  \text{ and }j^{\prime}>j. \label{48}%
\end{equation}
In words, the entry $k$ is a descent iff the row of the entry $k+1$ is below
the row of $k.$ We denote by $D\left(  Q\right)  \subset\lambda/\mu$ the set
of all descents cells of $Q$. Let
\begin{equation}
\mathsf{Des}\left(  Q\right)  =\left\{  i_{1},...,i_{l}\right\}  \label{49}%
\end{equation}
be the contents of descent cells of $Q.$ By definition,%
\begin{equation}
\mathsf{maj}\left(  Q\right)  =\sum_{k:Q\left(  k\right)  \in D\left(
Q\right)  }k\equiv\sum_{k\in\mathsf{Des}\left(  Q\right)  }k. \label{44}%
\end{equation}

We say that the SsYT $T$ agrees with SYT $Q,$ if

\begin{itemize}
\item for all $k$ we have%
\begin{equation}
T_{Q\left(  k\right)  }\leq T_{Q\left(  k+1\right)  }; \label{61}%
\end{equation}

\item if $k$ is a descent of $Q,$ then
\begin{equation}
T_{Q\left(  k\right)  }<T_{Q\left(  k+1\right)  }. \label{62}%
\end{equation}

\end{itemize}

For every SYT $Q$ of shape $\lambda/\mu$ we denote by $\mathcal{P}_{Q}%
^{Ss}\subset\mathcal{P}_{\lambda/\mu}^{Ss}$ the set of all SsYT $T$ agreeing
with $Q.$

\begin{proposition}
The sets $\mathcal{P}_{Q}^{Ss}$ do not intersect, and%

\[
\bigcup_{Q}\mathcal{P}_{Q}^{Ss}=\mathcal{P}_{\lambda/\mu}^{Ss}.
\]

\end{proposition}

\begin{proof}
Let $T$ be a SsYT. We are going to define the linear order $Q^{T}$ on
$\lambda/\mu,$ for which $T\in\mathcal{P}_{Q^{T}}^{Ss}.$

If
\begin{equation}
T_{ij}>T_{i^{\prime}j^{\prime}},\ \text{then we define }Q_{\left(  ij\right)
}^{T}>Q_{\left(  i^{\prime}j^{\prime}\right)  }^{T}, \label{67}%
\end{equation}
in order to satisfy $\left(  \ref{61}\right)  .$

Consider now the collection $I\left(  p\right)  $ of all cells $\left(
i,j\right)  $ for which the values $T_{ij}=p,\ p\geq0.$ Note that if $\left(
i^{\prime}j^{\prime}\right)  ,\left(  i,j\right)  \in I\left(  p\right)  $ are
two different cells, then $i^{\prime}\neq i,$ since otherwise the values
$T_{ij},T_{i^{\prime}j^{\prime}}$ have to be different. So the ordering
\begin{equation}
Q_{\left(  i^{\prime}j^{\prime}\right)  }^{T}>Q_{\left(  ij\right)  }%
^{T}\ \text{on }I\left(  p\right)  \ \text{iff }i^{\prime}>i \label{63}%
\end{equation}
is well defined, and thus the definition of the order $Q^{T}$ is competed.
From $\left(  \ref{67}\right)  ,$ $\left(  \ref{63}\right)  $ it follows that
$Q^{T}$ is SYT.

To see that $\left(  \ref{63}\right)  $ is the only possible way to define
$Q^{T}$ we have to consider the pairs of cells $\left(  i,j\right)  ,\left(
i^{\prime}j^{\prime}\right)  \in I\left(  p\right)  $ with $i^{\prime}>i$ and
$j^{\prime}>j.$ But for them the choice $\left(  \ref{63}\right)  $ is
mandatory, due to the condition $\left(  \ref{62}\right)  $ -- since the
opposite choice makes the cell $\left(  i,j\right)  $ a descent of $Q^{T}.$
\end{proof}

For two SsYT $T$ and $T^{\prime}$ we say that $T^{\prime}\succeq T$ iff
$T_{ij}^{\prime}\geq T_{ij}$ for each cell $ij\in\lambda/\mu.$ It is easy to
see that there exists a \textit{minimal} SsYT $\mathsf{t}_{\lambda/\mu},$ for
which we have $T\succeq\mathsf{t}_{\lambda/\mu}$ for all $T.$

For every $Q$ we denote by $\pi^{Q}\in\mathcal{P}_{Q}^{Ss}$ the smallest SsYT
in $\mathcal{P}_{Q}^{Ss}$ (in the $\succeq$ sense). We call $\pi^{Q}$ the
\textit{plinth} of $Q:$%
\[
\mathsf{p}\left(  Q\right)  =\pi^{Q}.
\]
We also say that if $T\in\mathcal{P}_{Q}^{Ss},$ then $\pi^{Q}$ is the plinth
of $T.$ Note that for different SYT $Q$ their plinths $\pi^{Q}$ are different,
due to the Proposition above. As an example, the above mentioned minimal SsYT
$\mathsf{t}_{\lambda/\mu}$ is a plinth -- namely, it is the plinth of the row
order $Q^{\operatorname{row}}$ on $\lambda/\mu.$ For the case $\lambda=\left(
4,4,4,3\right)  ,$ $\mu=\left(  2,1,1\right)  $ the row order is

$Q^{\operatorname{row}}=%
\begin{tabular}
[c]{|l|l|l|l|}\hline
&  & 1 & 2\\\hline
& 3 & 4 & 5\\\hline
& 6 & 7 & 8\\\hline
9 & 10 & 11 & \\\hline
\end{tabular}
\ $ and $\mathsf{t}_{\lambda/\mu}=\mathsf{p}\left(  Q^{\operatorname{row}%
}\right)  =%
\begin{tabular}
[c]{|l|l|l|l|}\hline
&  & 0 & 0\\\hline
& 0 & 1 & 1\\\hline
& 1 & 2 & 2\\\hline
0 & 2 & 3 & \\\hline
\end{tabular}
\ .$

The set of all plinths $\pi^{Q}\in\mathcal{P}_{\lambda/\mu}^{Ss}$ is denoted
by $\mathcal{N}_{\lambda/\mu}.$

\begin{theorem}
Let $\lambda/\mu$ be a skew Young diagram. For $T\in\mathcal{P}_{\lambda/\mu
}^{Ss}$ the correspondence
\begin{equation}
B^{Ss}:T\rightarrow\left(  \pi^{Q^{T}},Y\left(  T-\pi^{Q^{T}}\right)  \right)
\label{31}%
\end{equation}
is a bijection between $\mathcal{P}_{\lambda/\mu}^{Ss}$ and $\mathcal{N}%
_{\lambda/\mu}\times\mathcal{Y}_{\left\vert Y\right\vert },$ preserving the
volume. Here

$T-\pi^{Q^{T}}$ is the reverse plane partition of shape $\lambda/\mu$ (not
necessarily a SsYT), given by%
\[
\left(  T-\pi^{Q^{T}}\right)  _{ij}=T_{ij}-\left(  \pi^{Q^{T}}\right)  _{ij},
\]

$\mathcal{Y}_{\left\vert \lambda/\mu\right\vert }$ is the set of all Young
diagrams with $\left\vert \lambda/\mu\right\vert $ (non-negative) parts.

The inverse map $\bar{B}^{Ss}$ corresponds to the pair
\[
\left(  \pi^{Q},Y=\left\{  y_{1}\leq y_{2}\leq...\leq y_{\left\vert
\lambda/\mu\right\vert }\right\}  \right)  \in\mathcal{N}_{\lambda/\mu}%
\times\mathcal{Y}_{\left\vert \lambda/\mu\right\vert }%
\]
the SsYT $T,$ which is a sum of two functions on the cells of $\lambda/\mu:$
\begin{equation}
T=\pi^{Q}+f_{Y} \label{32}%
\end{equation}
-- the plinth $\pi_{ij}^{Q}$ and the function $\left(  f_{Y}\right)  _{ij}$,
$i,j\in\lambda/\mu,$ given by\newline%
\begin{equation}
\left(  f_{Y}\right)  _{Q\left(  k\right)  }=y_{k}. \label{33}%
\end{equation}
In particular, the relation $\left(  \ref{26}\right)  $ holds.
\end{theorem}

\begin{proof}
The volume preservation property follows immediately from the definitions
$\left(  \ref{31}\right)  ,$ $\left(  \ref{32}\right)  .$

By the definitions of the linear order $Q^{T}$ and of the plinths, all the
increments $T_{Q^{T}\left(  k+1\right)  }-T_{Q^{T}\left(  k\right)  }$ are
non-negative, and if the increment $\left(  \pi^{Q^{T}}\right)  _{Q^{T}\left(
k+1\right)  }-\left(  \pi^{Q^{T}}\right)  _{Q^{T}\left(  k\right)  }$ is
positive, then it is $1,$ and the increment $T_{Q^{T}\left(  k+1\right)
}-T_{Q^{T}\left(  k\right)  }$ is at least $1$. Therefore $T-\pi^{Q^{T}}$ is a
reverse plane partition, and the sequence $\left(  T-\pi^{Q^{T}}\right)
_{Q^{T}\left(  k\right)  }$ -- the reading diagram of $\left(  T-\pi^{Q^{T}%
}\right)  $ -- is non-decreasing sequence of non-negative integers, i.e. a
(1D) partition from $\mathcal{Y}_{\left\vert \lambda/\mu\right\vert }$,
denoted by $Y\left(  T-\pi^{Q^{T}}\right)  $ above.

Since the function $f_{Y}$ from $\left(  \ref{33}\right)  $ is non-decreasing
as a function of $k,$ the array $\left(  f_{Y}\right)  _{ij}$ of shape
$\lambda/\mu$ is a reverse plane partition, because $Q$ is a SYT. Therefore
the array $\left(  \pi^{Q}+f_{Y}\right)  _{ij}$ is also a reverse plane
partition, and since its increments are at least $1$ along every descent of
$Q,$ it is SsYT.

The relations $B^{Ss}\bar{B}^{Ss}=\bar{B}^{Ss}B^{Ss}=1$ follow from $\left(
\ref{32}\right)  ,$ $\left(  \ref{33}\right)  .$ Hence the maps $B^{Ss}$ and
$\bar{B}^{Ss}$ are bijections.
\end{proof}

\section{Jeu de taquin - a reminder}

Jeu de taquin is a remarkable equivalence relation among skew tableaux. A very
clear definition of it and its properties are given in the Appendix 1 (written
by S. Fomin) to the Stanley book \cite{St}. We will not repeat it here. But in
order to make the presentation self-contained, we will give an example of how
to play this absorbing game, in a comics manner. \newline The game starts with
a choice of a SYT $Q$ of skew shape $\lambda/\mu,$ for example

$%
\begin{tabular}
[c]{|l|l|l|l|}\hline
&  & 2 & 3\\\hline
& 1 & 5 & 6\\\hline
4 & 7 &  & \\\hline
8 &  &  & \\\hline
\end{tabular}
,$

\noindent plus a choice of one of the adjacent cells, marked by $u$ and $d$
(\textit{up} or \textit{down} cells):\newline

$%
\begin{tabular}
[c]{|l|l|l|l|l|}\hline
& $u$ & 2 & 3 & $d$\\\hline
$u$ & 1 & 5 & 6 & \\\hline
4 & 7 & $d$ &  & \\\hline
8 & $d$ &  &  & \\\hline
$d$ &  &  &  & \\\hline
\end{tabular}
.$

\noindent Let it be the $u$ cell $\left(  2,1\right)  :$

$%
\begin{tabular}
[c]{|l|l|l|l|}\hline
& $\ast$ & 2 & 3\\\hline
& 1 & 5 & 6\\\hline
4 & 7 &  & \\\hline
8 &  &  & \\\hline
\end{tabular}
.$

\noindent Then we do the following jeu de taquin moves:

$%
\begin{tabular}
[c]{|l|l|l|l|}\hline
& $\upuparrows$ & 2 & 3\\\hline
& 1 & 5 & 6\\\hline
4 & 7 &  & \\\hline
8 &  &  & \\\hline
\end{tabular}
\rightarrow%
\begin{tabular}
[c]{|l|l|l|l|}\hline
& $1$ & 2 & 3\\\hline
& $\leftleftarrows$ & 5 & 6\\\hline
4 & 7 &  & \\\hline
8 &  &  & \\\hline
\end{tabular}
\rightarrow%
\begin{tabular}
[c]{|l|l|l|l|}\hline
& $1$ & 2 & 3\\\hline
& 5 & $\leftleftarrows$ & 6\\\hline
4 & 7 &  & \\\hline
8 &  &  & \\\hline
\end{tabular}
\rightarrow%
\begin{tabular}
[c]{|l|l|l|l|}\hline
& $1$ & 2 & 3\\\hline
& 5 & 6 & \\\hline
4 & 7 &  & \\\hline
8 &  &  & \\\hline
\end{tabular}
,$

\noindent each time moving the \textit{smallest} among the East and the South
neighbor to the vacant cell (provided it has both neighbors, otherwise we use
the one which is present), until the newborn vacant cell has no East and no
South neighbor. It turns out that the result is again a SYT (of different
shape). \newline

That completes one \textit{slide} of jeu de taquin. The above slide will be
called a \textit{slide of} $Q$ \textit{into the cell} $\left(  2,1\right)  .$

One can do the same with initial cell being a $d$ cell; the only difference is
that one has to replace the words \textit{smallest} \textit{among the East and
the South }by \textit{largest} \textit{among the West and the North. }\newline

SYTs $Q$ and $Q^{\prime}$ are called jeu de taquin equivalent, $Q\overset
{jdt}{\sim}Q^{\prime},$ if one can be obtained from another by a sequence of
jeu de taquin slides (see the Definition A1.3.2 in \cite{St}).

Note that after one $u$-slide the number of empty NW boxes decreases by one.
We can repeat the $u$ slides till there is no more empty NW boxes. In our case
that requires two more slides:%

\begin{equation}%
\begin{tabular}
[c]{|l|l|l|l|}\hline
& 1 & 2 & 3\\\hline
$\ast$ & 5 & 6 & \\\hline
4 & 7 &  & \\\hline
8 &  &  & \\\hline
\end{tabular}
\rightarrow%
\begin{tabular}
[c]{|l|l|l|l|}\hline
& 1 & 2 & 3\\\hline
$4$ & 5 & 6 & \\\hline
& 7 &  & \\\hline
8 &  &  & \\\hline
\end{tabular}
\rightarrow%
\begin{tabular}
[c]{|l|l|l|l|}\hline
& 1 & 2 & 3\\\hline
$4$ & 5 & 6 & \\\hline
7 &  &  & \\\hline
8 &  &  & \\\hline
\end{tabular}
, \label{72}%
\end{equation}
and%

\begin{equation}%
\begin{tabular}
[c]{|l|l|l|l|}\hline
$\ast$ & 1 & 2 & 3\\\hline
$4$ & 5 & 6 & \\\hline
7 &  &  & \\\hline
8 &  &  & \\\hline
\end{tabular}
\rightarrow%
\begin{tabular}
[c]{|l|l|l|l|}\hline
1 &  & 2 & 3\\\hline
$4$ & 5 & 6 & \\\hline
7 &  &  & \\\hline
8 &  &  & \\\hline
\end{tabular}
\rightarrow%
\begin{tabular}
[c]{|l|l|l|l|}\hline
1 & 2 &  & 3\\\hline
$4$ & 5 & 6 & \\\hline
7 &  &  & \\\hline
8 &  &  & \\\hline
\end{tabular}
\rightarrow%
\begin{tabular}
[c]{|l|l|l|l|}\hline
1 & 2 & 3 & $\ $\\\hline
$4$ & 5 & 6 & \\\hline
7 &  &  & \\\hline
8 &  &  & \\\hline
\end{tabular}
. \label{73}%
\end{equation}

\noindent The resulting SYT is of straight shape. Clearly, to get to the
straight shape via jdt slides one needs to choose a sequence of the $u$ cells,
so in principle the resulting straight SYT might depend on this sequence. But
it is not the case, and in fact even a stronger statement holds:

\begin{theorem}
\label{main} (A1.3.4 in the Appendix, \cite{St}) Each jeu de taquin
equivalence class contains exactly one straight-shape SYT.
\end{theorem}

For any SYT $Q$ we denote by $\mathbf{\ulcorner}Q$ the jdt equivalent SYT of
straight shape.

\section{Sch\"{u}tzenberger involution - a reminder}

In this section we remind the reader about the definition of the
Sch\"{u}tzenberger involution, $Sch$, which acts on straight-shaped SYT.
Again, the definition can be found in the Appendix 1 to \cite{St}, so we just
illustrate it with an example. In fact, our definition is a slight
reformulation of that in \cite{St}. In our version, for every SYT $Q$ of shape
$\lambda$, $\left\vert \lambda\right\vert =n,$ we define a sequence
$Q_{1},...,Q_{n}$ of SYTs of the same shape $\lambda,$ and then the SYT
$Sch\left(  Q\right)  $ is by definition the last SYT $Q_{n}.$ \newline

Let $Q=%
\begin{tabular}
[c]{|l|l|l|}\hline
1 & 2 & 7\\\hline
3 & 5 & \\\hline
4 & 6 & \\\hline
\end{tabular}
.$

$1.$ Subtract $1$ from each cell of $Q,$ and put our symbol $u$ instead of
$0.$ We get $%
\begin{tabular}
[c]{|l|l|l|}\hline
$u$ & 1 & 6\\\hline
2 & 4 & \\\hline
3 & 5 & \\\hline
\end{tabular}
\ .$

$2.$\textit{ }Do a jdt slide to the cell $u.$ We get $%
\begin{tabular}
[c]{|l|l|l|}\hline
$1$ & 4 & 6\\\hline
2 & 5 & \\\hline
3 &  & \\\hline
\end{tabular}
.$

$3.$ We put back the \textit{lost} value $7$ into the above table; it goes to
the newly vacated cell. In our example it is the cell $\left(  2,3\right)  ,$
so we are getting $%
\begin{tabular}
[c]{|l|l|l|}\hline
$1$ & 4 & 6\\\hline
2 & 5 & \\\hline
3 & $\mathbf{7}$ & \\\hline
\end{tabular}
\ .$ This is $Q_{1}=$ $%
\begin{tabular}
[c]{|l|l|l|}\hline
$1$ & 4 & 6\\\hline
2 & 5 & \\\hline
3 & $\mathbf{7}$ & \\\hline
\end{tabular}
\ .$ The number $\mathbf{7}$ is marked \textbf{bold}; we will see now the
reason for it.\newline

$1^{\prime}.$ Subtract again $1$ from each cell of $Q_{1},$ which is
\textbf{not} marked \textbf{bold}. In fact, the bold cells will not change any
more. This subtraction results in $%
\begin{tabular}
[c]{|l|l|l|}\hline
$u$ & 3 & 5\\\hline
1 & 4 & \\\hline
2 & $\mathbf{7}$ & \\\hline
\end{tabular}
\ .$

$2^{\prime}.$ We do a jdt slide to the cell $u$ of the SYT made by six regular
cells, i.e. excluding the cell containing $\mathbf{7}$. We get $%
\begin{tabular}
[c]{|l|l|l|}\hline
$1$ & 3 & 5\\\hline
2 & 4 & \\\hline
& $\mathbf{7}$ & \\\hline
\end{tabular}
.$

$3^{\prime}.$ We put back the value $6,$ which was\textit{ lost}, it goes to
the newly vacated cell. We get $Q_{2}=%
\begin{tabular}
[c]{|l|l|l|}\hline
$1$ & 3 & 5\\\hline
2 & 4 & \\\hline
$\mathbf{6}$ & $\mathbf{7}$ & \\\hline
\end{tabular}
\ .$ Again, $\mathbf{6}$ is in bold, and it will not change any more.

$1^{\prime\prime}.$ The next cycle involves SYT formed only by 5 regular
cells. We subtract $1$ from each of them, getting $%
\begin{tabular}
[c]{|l|l|l|}\hline
$u$ & 2 & 4\\\hline
1 & 3 & \\\hline
$\mathbf{6}$ & $\mathbf{7}$ & \\\hline
\end{tabular}
.$

$2^{\prime\prime}.$ A jdt slide to the cell $u$ brings us to $%
\begin{tabular}
[c]{|l|l|l|}\hline
$1$ & 2 & 4\\\hline
3 &  & \\\hline
$\mathbf{6}$ & $\mathbf{7}$ & \\\hline
\end{tabular}
.$

$3^{\prime\prime}.$ $5$ reappears in \textbf{bold }in the cell $\left(
2,2\right)  $ just vacated: $Q_{3}=%
\begin{tabular}
[c]{|l|l|l|}\hline
$1$ & 2 & 4\\\hline
3 & $\mathbf{5}$ & \\\hline
$\mathbf{6}$ & $\mathbf{7}$ & \\\hline
\end{tabular}
.$

Proceeding, we get $Q_{4}=%
\begin{tabular}
[c]{|l|l|l|}\hline
$1$ & 3 & $\mathbf{4}$\\\hline
2 & $\mathbf{5}$ & \\\hline
$\mathbf{6}$ & $\mathbf{7}$ & \\\hline
\end{tabular}
,$ $Q_{5}=%
\begin{tabular}
[c]{|l|l|l|}\hline
$1$ & 2 & $\mathbf{4}$\\\hline
$\mathbf{3}$ & $\mathbf{5}$ & \\\hline
$\mathbf{6}$ & $\mathbf{7}$ & \\\hline
\end{tabular}
,$ $Q_{6}=%
\begin{tabular}
[c]{|l|l|l|}\hline
$1$ & $\mathbf{2}$ & $\mathbf{4}$\\\hline
$\mathbf{3}$ & $\mathbf{5}$ & \\\hline
$\mathbf{6}$ & $\mathbf{7}$ & \\\hline
\end{tabular}
,$ $Q_{7}=%
\begin{tabular}
[c]{|l|l|l|}\hline
$\mathbf{1}$ & $\mathbf{2}$ & $\mathbf{4}$\\\hline
$\mathbf{3}$ & $\mathbf{5}$ & \\\hline
$\mathbf{6}$ & $\mathbf{7}$ & \\\hline
\end{tabular}
,$ so $Sch\left(  Q\right)  =%
\begin{tabular}
[c]{|l|l|l|}\hline
$1$ & $2$ & $4$\\\hline
$3$ & $5$ & \\\hline
$6$ & $7$ & \\\hline
\end{tabular}
.$

The transformation presented is an \textit{involution}:

\begin{theorem}
For every shape $\lambda$ and every SYT $Q$ of shape $\lambda$ we have%
\[
Sch\left(  Sch\left(  Q\right)  \right)  =Q.
\]

\end{theorem}

For the proof, see \cite{St}, Proposition A1.4.2.

\section{\textbf{Plinth=Sch\"{u}tzenberger(maj): the case of }straight-shaped
SYT.}

In this section we will prove that for a straight-shaped SYT $Q$ of shape
$\lambda$ we have

\begin{theorem}
\label{mainT}%
\begin{equation}
\left\vert \mathsf{p}\left(  Q\right)  \right\vert =\mathsf{maj}\left(
Sch\left(  Q\right)  \right)  . \label{41}%
\end{equation}

\end{theorem}

\begin{proof}
Let $Q$ be a standard Young tableau (SYT), and $\pi^{Q}$ be the plinth of $Q.$
Let $\mathsf{Des}\left(  Q\right)  =\left\{  i_{1},...,i_{l}\right\}  $ be the
contents of descent cells, so $\mathsf{maj}\left(  Q\right)  =i_{1}+...+i_{l}$
is the sum of the descent values.

\vskip.2cm The volume%
\begin{align}
&  \mathsf{vol}\left(  \pi^{Q}\right) \nonumber\\
&  =0\cdot i_{1}+1\cdot(i_{2}-i_{1})+2\cdot(i_{3}-i_{2})+\ldots+(l-1)\cdot
(i_{l}-i_{l-1})+l\cdot(n-i_{l})\label{42}\\
&  =n-i_{1}+...+n-i_{l}=nl-\left(  i_{1}+...+i_{l}\right)
=nl-\text{$\mathsf{maj}$}\left(  Q\right)  \ .\nonumber
\end{align}

Next, we use two facts concerning the Robinson-Schensted correspondence:

\begin{itemize}
\item Let $w$ be any permutation, such that the corresponding
Robinson-Schensted pair of SYT-s is $\left(  P,Q\right)  ,$ i.e. our $Q$ is a
recording tableau of $w.$ Then $\mathsf{Des}\left(  Q\right)  =\mathsf{Des}%
\left(  w\right)  .$

\item Consider the permutation $w^{\prime}$ obtained by writing $w$ backwards
and replacing the entries $1,2,...,n$ by $n,n-1,...,1$ ($w\rightarrow
w^{\prime}$ is an involution\footnote{Let $\omega_{0}:=(1,n)(2,n-1)\ldots$ be
the longest element in the Coxeter group $S_{n}$. Then $w^{\prime}=\omega
_{0}w\omega_{0}^{-1}$.}). Then Sch\"{u}tzenberger involution $Sch\left(
Q\right)  $ is the recording tableaux $Q^{\prime}$ of $w^{\prime}.$
\end{itemize}

The descent set of $w^{\prime}$ is the reversal of $\mathsf{Des}\left(
w\right)  :$ $(i,i+1)$ is a descent in $w$ if and only if $(n-i,n+1-i)$ is a
descent in $w^{\prime}$. Hence
\[
\text{$\mathsf{maj}$}\left(  Q^{\prime}\right)  =n-i_{1}+...+n-i_{l}%
=\mathsf{vol}\left(  \pi^{Q}\right)  =\left\vert \mathsf{p}\left(  Q\right)
\right\vert \ .
\]

\end{proof}

In particular, $\mathsf{maj}\left(  Q\right)  +\mathsf{maj}\left(  Q^{\prime
}\right)  =nl.$

\begin{remark}
After conjecturing the relation $\left(  \ref{41}\right)  $ and checking it
numerically in several cases we got a message from Professor S. Fomin, who
explained to us that the expression $\left(  \ref{42}\right)  $ is nothing
else but the value $\mathsf{maj}\left(  w^{\prime}\right)  .$ We are grateful
to him for this remark.
\end{remark}

\section{Extending the Sch\"{u}tzenberger involution}

After realizing that the two statistics -- $\left\vert \mathsf{p}\left(
Q\right)  \right\vert $ and $\mathsf{maj}\left(  Q\right)  $ -- are
equidistributed on the set of SYT $Q$-s for each straight shape $\lambda,$ we
have checked numerically that the same holds for some skew shapes $\lambda
/\mu.$ Getting the positive answer we got the idea that the Sch\"{u}tzenberger
involution can be extended to the set of SYT-s of skew shape, so that the
relation $\left(  \ref{41}\right)  $ still holds. This is indeed the case, as
we will show below. To do this we will use some beautiful results of Haiman,
\cite{H}. In particular, we will use his \textit{dual (to jeu de taquin)
equivalence relation }between the SYT of the same shape.

Let $Q$ be a SYT of the shape $\lambda/\mu,$ and let $\left(  c_{1}%
,...,c_{l}\right)  $ be a sequence of cells for which it is meaningful to form
a sequence $Q_{0}=Q,Q_{1},...,Q_{l}$ of SYT-s, such that each $Q_{j}$ is a
result of slide of $Q_{j-1}$ into the cell $c_{j}.$ For example, consider two
SYT $X$ and $Y,$ of the shapes $sh\left(  X\right)  ,$ $sh\left(  Y\right)  ,$
such that the Young diagrams $sh\left(  X\right)  ,$ $sh\left(  Y\right)
\subset\mathbb{R}^{2}$ do not intersect and their union $sh\left(  X\right)
\cup$ $sh\left(  Y\right)  $ is again a Young diagram, while $X$ lies to the
NW of $Y.$ In that case we will say that $sh\left(  Y\right)  $
\textbf{extends} $sh\left(  X\right)  .$ Then the sequence of cells forming
$sh\left(  Y\right)  ,$ taken in the order $Y$ can be used to slide the SYT
$X$ in SE direction. The result will be denoted by $j_{Y}\left(  X\right)  .$
In the corresponding way we can move $Y$ in the NW direction, using $X.$ We
will get SYT, denoted by $j^{X}\left(  Y\right)  .$ The sequence of cells
vacated as we form $j_{Y}\left(  X\right)  $ defines a SYT which we denote by
$\left[  V:j_{Y}\left(  X\right)  \right]  ;$ this SYT is located to the NW of
the SYT $j_{Y}\left(  X\right)  .$ The SYT $\left[  V:j^{X}\left(  Y\right)
\right]  ,$ located to the SE of $j^{X}\left(  Y\right)  ,$ is defined in
analogous way. Clearly,%
\begin{equation}
j_{\left[  V:j^{X}\left(  Y\right)  \right]  }j^{X}\left(  Y\right)
=Y,\ j^{\left[  V:j_{Y}\left(  X\right)  \right]  }j_{Y}\left(  X\right)  =X.
\label{56}%
\end{equation}

\begin{lemma}
(\cite{H}) Let $X$, $Y$ be a pair of SYT, and $sh\left(  Y\right)  $ extends
$sh\left(  X\right)  .$

Then $\left[  V:j^{X}\left(  Y\right)  \right]  =j_{Y}\left(  X\right)  ,$
$\left[  V:j_{Y}\left(  X\right)  \right]  =j^{X}\left(  Y\right)  .$
\end{lemma}

\begin{definition}
(\cite{H}) Let $P,Q$ be a pair of SYT of the same (skew) shape. Suppose that
every sequence $\left(  c_{1},...,c_{l}\right)  $ which is a sequence of
slides for both $P$ and $Q$ yields two tableaux of the same shape when applied
to $P,$ resp. $Q.$ Then the SYT-s $P$ and $Q$ are said to be \textit{dual
equivalent, }$P\overset{Djdt}{\sim}Q.$
\end{definition}

\begin{theorem}
(\cite{H}) Two SYT $P,Q$ are dual equivalent, $P\overset{Djdt}{\sim}Q,$ if and
only if there exists a pair $X,Y$ of SYT of equal \textbf{straight} shape
$sh\left(  X\right)  =sh\left(  Y\right)  ,$ such that $P\overset{jdt}{\sim}X$
and $Q\overset{jdt}{\sim}Y.$ In particular, any two SYT $X,Y$ of the same
\textbf{straight} shape, $sh\left(  X\right)  =sh\left(  Y\right)  ,$ are dual
equivalent, $X\overset{Djdt}{\sim}Y.$

Suppose that two SYT $P,Q$ have the same shape $sh\left(  P\right)  =sh\left(
Q\right)  $ and are dual equivalent, $P\overset{Djdt}{\sim}Q.$ Suppose also
that the shape $sh\left(  P\right)  =sh\left(  Q\right)  $ extends the shape
$sh\left(  X\right)  $ of the SYT $X.$ Then%
\begin{equation}
j_{P}\left(  X\right)  =j_{Q}\left(  X\right)  , \label{52}%
\end{equation}
and
\begin{equation}
\left[  V:j_{P}\left(  X\right)  \right]  \overset{Djdt}{\sim}\left[
V:j_{Q}\left(  X\right)  \right]  . \label{51}%
\end{equation}

\end{theorem}

The above information enables us to extend the Sch\"{u}tzenberger involution,
defined initially for the straight shaped SYT, to the skew shaped SYT.

\begin{definition}
Let the SYT $Q$ be of the shape $sh\left(  Q\right)  =\lambda/\mu.$ Choose
some SYT $P$ of shape $sh\left(  P\right)  =\mu.$ We define%
\begin{equation}
\widetilde{Sch}\left(  Q\right)  =j_{\left[  V:j^{P}\left(  Q\right)  \right]
}\left(  Sch\left(  j^{P}\left(  Q\right)  \right)  \right)  . \label{54}%
\end{equation}

\end{definition}

\begin{theorem}
The SYT $\widetilde{Sch}\left(  Q\right)  $ in $\left(  \ref{54}\right)  $ is
well-defined, i.e. it does not depend on the particular choice of the SYT $P$
of shape $\mu.$

The SYT $\widetilde{Sch}\left(  Q\right)  $ has the same shape as $Q,$ and%
\[
\widetilde{Sch}\left(  \widetilde{Sch}\left(  Q\right)  \right)  =Q.
\]
For $Q$ of straight shape $\widetilde{Sch}\left(  Q\right)  =Sch\left(
Q\right)  .$
\end{theorem}

\begin{proof}
Since the shape of the SYT $j^{P}\left(  Q\right)  $ is a straight shape, its
Sch\"{u}tzenberger involution $Sch\left(  j^{P}\left(  Q\right)  \right)  $ is
already defined above. As we know from the Theorem \ref{main}, the SYT
$j^{P}\left(  Q\right)  $ does not depend on the SYT $P$ (which eliminates the
second appearance of $P$ in $\left(  \ref{54}\right)  $). This might not be
true, however, for the SYT $\left[  V:j^{P}\left(  Q\right)  \right]  .$ For
different $P,P^{\prime}$ we know only that
\begin{equation}
\left[  V:j^{P}\left(  Q\right)  \right]  \overset{Djdt}{\sim}\left[
V:j^{P^{\prime}}\left(  Q\right)  \right]  \label{55}%
\end{equation}
-- indeed, $P$ and $P^{\prime}$ are straight and of the same shape, so
$P\overset{Djdt}{\sim}P^{\prime}$ and we can use $\left(  \ref{51}\right)  .$
However, the relation $\left(  \ref{55}\right)  $ is sufficient, since due to
$\left(  \ref{52}\right)  ,$ we have that
\[
j_{\left[  V:j^{P}\left(  Q\right)  \right]  }\left(  Sch\left(  j^{P}\left(
Q\right)  \right)  \right)  =j_{\left[  V:j^{P^{\prime}}\left(  Q\right)
\right]  }\left(  Sch\left(  j^{P}\left(  Q\right)  \right)  \right)
=j_{\left[  V:j^{P^{\prime}}\left(  Q\right)  \right]  }\left(  Sch\left(
j^{P^{\prime}}\left(  Q\right)  \right)  \right)  ,
\]
so the SYT $\widetilde{Sch}\left(  Q\right)  $ is well-defined.

Let us now check that $\widetilde{Sch}\left(  \widetilde{Sch}\left(  Q\right)
\right)  =Q.$ To see this let us introduce the notation $\Delta$ for the
involution $\left(  P,Q\right)  \overset{\Delta}{\rightarrow}\left(
j^{P}\left(  Q\right)  ,\left(  j_{Q}\left(  P\right)  \right)  \right)  $,
where SYT $Q$ extends the straight SYT $P.$ Then we find the SYT
$\widetilde{Sch}\left(  Q\right)  $ to appear at the end of the transformation
string $\mathcal{T}$:\newline%
\[
\mathcal{T\ }:\ \left(  P,Q\right)  \overset{\Delta}{\rightarrow}\left(
j^{P}\left(  Q\right)  ,\left(  j_{Q}\left(  P\right)  \right)  \right)
\overset{Sch}{\rightarrow}\left(  Sch\left[  j^{P}\left(  Q\right)  \right]
,\left(  j_{Q}\left(  P\right)  \right)  \right)  \overset{\Delta}%
{\rightarrow}\left(  P,\widetilde{Sch}\left(  Q\right)  \right)  .\newline%
\]
We have used here the fact that $j^{P}\left(  Q\right)  \overset{Djdt}{\sim
}Sch\left[  j^{P}\left(  Q\right)  \right]  ,$ since both SYTs are straight
and of the same shape. From here we see that the shape of SYT $\widetilde
{Sch}\left(  Q\right)  $ is that of $Q,$ since the shape of the union of the
pair of two tables at each step of $\mathcal{T\ }$is not changing. Repeating
$\mathcal{T}\ $once more we have%
\begin{align*}
\mathcal{T\ }  &  :\ \left(  P,\widetilde{Sch}\left(  Q\right)  \right)
\overset{\Delta}{\rightarrow}\left(  j^{P}\left(  \widetilde{Sch}\left(
Q\right)  \right)  ,\left(  j_{\widetilde{Sch}\left(  Q\right)  }\left(
P\right)  \right)  \right)  \overset{1}{=}\left(  Sch\left[  j^{P}\left(
Q\right)  \right]  ,\left(  j_{Q}\left(  P\right)  \right)  \right) \\
&  \overset{Sch}{\rightarrow}\left(  Sch\left[  Sch\left[  j^{P}\left(
Q\right)  \right]  \right]  ,\left(  j_{Q}\left(  P\right)  \right)  \right)
\overset{2}{=}\left(  j^{P}\left(  Q\right)  ,\left(  j_{Q}\left(  P\right)
\right)  \right)  \overset{\Delta}{\rightarrow}\left(  P,Q\right)  .
\end{align*}
Indeed, the first equality holds since $\Delta$ is an involution, and the
second -- since $Sch$ is.
\end{proof}

\section{Plinth=Sch\"{u}tzenberger(maj): general case}

Here we prove the Theorem \ref{mainT} for the SYT $Q$ of arbitrary skew shape
$\lambda/\mu.$ The proof is based on the observation that the set of descents
values (see $\left(  \ref{48}\right)  $) stays unchanged as we perform the
steps of jdt slides. In order to see this we have to generalize slightly the
notion of descent. Indeed, we defined the descents above for 2D arrays of
integers having the shape of a skew Young diagram. However, one can extend the
definition in an obvious way for arrays $Q$ having the shape of a skew Young
diagram \textit{with holes }(which have already appeared in $\left(
\ref{72}\right)  ,$ $\left(  \ref{73}\right)  $)\textit{,} by repeating the
definition $\left(  \ref{48}\right)  ,$ that $k$ is a descent value of $Q$ iff
the cell $Q\left(  k+1\right)  $ belongs to the lower row than the row of the
cell $Q\left(  k\right)  .$

\begin{lemma}
Let $Q^{\prime}$ be a SYT (may be with a hole), and $Q^{\prime\prime}$ be the
SYT obtained from $Q^{\prime}$ by a single jdt move. Then $\mathsf{Des}\left(
Q^{\prime}\right)  =\mathsf{Des}\left(  Q^{\prime\prime}\right)  .$
\end{lemma}

\begin{proof}
Depending on the shape of the SYT $Q^{\prime}$ and the location of the cell
$\ast$ in it, there are four different cases of the initial configuration of
the jdt step:

$%
\begin{tabular}
[c]{|l|l|l|}\hline
$\cdot$ & $\cdot$ & $\cdot$\\\hline
$\cdot$ & $\ast$ & $s$\\\hline
$\cdot$ &  & \\\hline
\end{tabular}
,\
\begin{tabular}
[c]{|l|l|l|}\hline
$\cdot$ & $\cdot$ & $\cdot$\\\hline
$\cdot$ & $\ast$ & $s$\\\hline
$\cdot$ & $t$ & $\cdot$\\\hline
\end{tabular}
,\
\begin{tabular}
[c]{|l|l|l|}\hline
$\cdot$ & $\cdot$ & $\cdot$\\\hline
$\cdot$ & $\ast$ & \\\hline
$\cdot$ & $s$ & \\\hline
\end{tabular}
,\ $and $%
\begin{tabular}
[c]{|l|l|l|}\hline
$\cdot$ & $\cdot$ & $\cdot$\\\hline
$\cdot$ & $\ast$ & $t$\\\hline
$\cdot$ & $s$ & $\cdot$\\\hline
\end{tabular}
.$ Here $s<t$ are integers, and in the first case the cell $\left(
2,3\right)  $ is not part of $Q^{\prime},$ while in the third case the cell
$\left(  3,2\right)  $ is not part of it.

\textbf{1. }In the first case the jdt step is

$%
\begin{tabular}
[c]{|l|l|l|}\hline
$\cdot$ & $\cdot$ & $\cdot$\\\hline
$\cdot$ & $\ast$ & $s$\\\hline
$\cdot$ &  & \\\hline
\end{tabular}
\rightarrow%
\begin{tabular}
[c]{|l|l|l|}\hline
$\cdot$ & $\cdot$ & $\cdot$\\\hline
$\cdot$ & $s$ & $\ast$\\\hline
$\cdot$ &  & \\\hline
\end{tabular}
,$ and since no cell moves vertically, the set $\mathsf{Des}$ does not change.

\textbf{2.} The same applies to the second case, where we have

$%
\begin{tabular}
[c]{|l|l|l|}\hline
$\cdot$ & $\cdot$ & $\cdot$\\\hline
$\cdot$ & $\ast$ & $s$\\\hline
$\cdot$ & $t$ & $\cdot$\\\hline
\end{tabular}
$ $\rightarrow%
\begin{tabular}
[c]{|l|l|l|}\hline
$\cdot$ & $\cdot$ & $\cdot$\\\hline
$\cdot$ & $s$ & $\ast$\\\hline
$\cdot$ & $t$ & $\cdot$\\\hline
\end{tabular}
.$

\textbf{3. }Consider the jdt step

$%
\begin{tabular}
[c]{|l|l|l|}\hline
$\cdot$ & $\cdot$ & $\cdot$\\\hline
$\cdot$ & $\ast$ & \\\hline
$\cdot$ & $s$ & \\\hline
\end{tabular}
\rightarrow%
\begin{tabular}
[c]{|l|l|l|}\hline
$\cdot$ & $\cdot$ & $\cdot$\\\hline
$\cdot$ & $s$ & \\\hline
$\cdot$ & $\ast$ & \\\hline
\end{tabular}
.$ Note that the only two values which we have to worry about are $s$ and
$s-1.$

If $s$ is in $\mathsf{Des}\left(  Q^{\prime}\right)  ,$ then $s$ is in
$\mathsf{Des}\left(  Q^{\prime\prime}\right)  ,$ since $s$ is ascending, and
therefore remains a descent.

Note that the cell $Q^{\prime}\left(  s+1\right)  $ does not belong to the row
of the cell $Q^{\prime}\left(  s\right)  $ -- indeed, it cannot be in this row
and to the left of $Q^{\prime}\left(  s\right)  $ since $s+1>s,$ while
$Q^{\prime}\left(  s\right)  $ is the rightmost cell of that row. Therefore if
$s\notin\mathsf{Des}\left(  Q^{\prime}\right)  ,$ then $s\notin\mathsf{Des}%
\left(  Q^{\prime\prime}\right)  .$

If the cell $Q^{\prime}\left(  s-1\right)  $ is not a descent, then the upward
move of the cell $Q^{\prime}\left(  s\right)  $ cannot make it to be one.

If the cell $Q^{\prime}\left(  s-1\right)  $ is two or more rows higher than
the cell $Q^{\prime}\left(  s\right)  ,$ then the value $s-1$ is a descent for
$Q^{\prime}$ as well as for $Q^{\prime\prime}.$

Finally, the cell $Q^{\prime}\left(  s-1\right)  $ cannot be one row higher
than $Q^{\prime}\left(  s\right)  $ -- indeed, all the values to the left of
$\ast$ in $Q^{\prime}$ are $\leq s-2,$ since for the fragment $%
\begin{tabular}
[c]{|l|l|l|}\hline
$\cdot$ & $\cdot$ & $\cdot$\\\hline
$x$ & $\ast$ & \\\hline
$y$ & $s$ & \\\hline
\end{tabular}
$ of $Q^{\prime}$ we have $x<y<s.$ So $s-1\in\mathsf{Des}\left(  Q^{\prime
}\right)  \ $iff $s-1\in\mathsf{Des}\left(  Q^{\prime\prime}\right)  .$

\textbf{4. }The remaining jdt step to consider is

$%
\begin{tabular}
[c]{|l|l|l|}\hline
$\cdot$ & $\cdot$ & $\cdot$\\\hline
$\cdot$ & $\ast$ & $t$\\\hline
$\cdot$ & $s$ & $\cdot$\\\hline
\end{tabular}
\rightarrow%
\begin{tabular}
[c]{|l|l|l|}\hline
$\cdot$ & $\cdot$ & $\cdot$\\\hline
$\cdot$ & $s$ & $t$\\\hline
$\cdot$ & $\ast$ & $\cdot$\\\hline
\end{tabular}
.$

As in \textbf{3}, if $s$ is in $\mathsf{Des}\left(  Q^{\prime}\right)  ,$ then
$s$ is in $\mathsf{Des}\left(  Q^{\prime\prime}\right)  ,$ since $s$ has ascended.

Again, as in \textbf{3}, the cell $Q^{\prime}\left(  s+1\right)  $ does not
belong to the row of the cell $Q^{\prime}\left(  s\right)  $ -- it cannot be
in this row and to the left of $Q^{\prime}\left(  s\right)  $ since $s+1>s,$
while every entry $x$ in that row to the right of $Q^{\prime}\left(  s\right)
$ satisfy $t<x,$ so $x\geq s+2,$ thus if $s\notin\mathsf{Des}\left(
Q^{\prime}\right)  $ then $s\notin\mathsf{Des}\left(  Q^{\prime\prime}\right)
.$

If the cell $Q^{\prime}\left(  s-1\right)  $ is not a descent, then the upward
move of the cell $Q^{\prime}\left(  s\right)  $ cannot make it to be one.

If the cell $Q^{\prime}\left(  s-1\right)  $ is two or more rows higher than
the cell $Q^{\prime}\left(  s\right)  ,$ then $s-1\in\mathsf{Des}\left(
Q^{\prime}\right)  $ and $s-1\in\mathsf{Des}\left(  Q^{\prime\prime}\right)
.$ Finally, the cell $Q^{\prime}\left(  s-1\right)  $ cannot be one row higher
than $Q^{\prime}\left(  s\right)  $ -- indeed, $t>s,$ while for the fragment $%
\begin{tabular}
[c]{|l|l|l|}\hline
$\cdot$ & $\cdot$ & $\cdot$\\\hline
$x$ & $\ast$ & $t$\\\hline
$y$ & $s$ & $\cdot$\\\hline
\end{tabular}
$ of $Q^{\prime}$ we have $x<y<s.$
\end{proof}

\begin{theorem}
For SYT $Q$ of the skew shape $\lambda/\mu$ we have
\[
\left\vert \mathsf{p}\left(  Q\right)  \right\vert =\mathsf{maj}\left(
\widetilde{Sch}\left(  Q\right)  \right)  .
\]

\end{theorem}

The proof follows immediately from the Theorem \ref{mainT} and the last Lemma.
Indeed, both the volume of the plinth of a SYT $Q$ and its index
$\mathsf{maj}$ are determined by the set $\mathsf{Des}\left(  Q\right)  $ (see
$\left(  \ref{42}\right)  $). Since $\mathsf{Des}\left(  \ast\right)  $ does
not change after jdt slides, we have $\left\vert \mathsf{p}\left(  Q\right)
\right\vert =\left\vert \mathsf{p}\left(  \mathbf{\ulcorner}Q\right)
\right\vert ,$ while $\mathbf{\ulcorner}\left(  \widetilde{Sch}\left(
Q\right)  \right)  =Sch\left(  \mathbf{\ulcorner}Q\right)  $ by construction,
and we can use the Theorem \ref{mainT} to conclude that%
\[
\left\vert \mathsf{p}\left(  Q\right)  \right\vert =\left\vert \mathsf{p}%
\left(  \mathbf{\ulcorner}Q\right)  \right\vert =\mathsf{maj}\left(
Sch\left(  \mathbf{\ulcorner}Q\right)  \right)  =\mathsf{maj}\left(
\widetilde{Sch}\left(  Q\right)  \right)  .
\]

\vskip.4cm \textbf{Acknowledgements.} The work of S.S. was supported by the
RSF under project 23-11-00150.

We thank Professor S. Fomin for his valuable remarks and suggestions.

\end{document}